\documentclass[12pt]{article}
\usepackage[cp1251]{inputenc}
\usepackage[T2A]{fontenc}
\usepackage[english]{babel}
\usepackage{blindtext}
\usepackage[margin=2.5cm,left=2.5cm,includefoot]{geometry}
\usepackage[mathcal]{eucal}
\usepackage{amsfonts}
\usepackage{amssymb}
\usepackage{ amssymb }
\usepackage{amsmath}
\usepackage{hyperref}
\usepackage{abstract}
\usepackage[symbol]{footmisc}

\usepackage{amsthm}
\usepackage{amscd}
\usepackage{mathrsfs}
\usepackage{ dsfont }
\usepackage{xcolor}
\usepackage{blindtext}

\usepackage{amsfonts}
\usepackage{amssymb}
\usepackage{ amssymb }
\usepackage{amsmath}
\usepackage{abstract}
\usepackage[symbol]{footmisc}

\theoremstyle{plain}

\newtheorem{lem}{Lemma}[section]

\newtheorem{theor}{Theorem}

\newtheorem{theorem}{Theorem}

\newtheorem{coroll}{Corollary}

\begin{document}

\begin{center}
\textsc{\Large Rational approximations to two irrational numbers\footnote[2]{This research is supported by the Russian Science Foundation under grant [19-11-00065] and performed in Khabarovsk Division of the Institute for Applied Mathematics, Far Eastern Branch, Russian Academy of Sciences.}.\\ }
\large Nikita Shulga
\end{center}

\begin{abstract}
For real $\xi$ we consider the irrationality measure function $\psi_\xi(t) = \min_{1\leqslant q \leqslant t, q\in\mathbb{Z}} || q\xi ||$. We prove that in the case $\alpha\pm\beta\notin\mathbb{Z}$ there exist arbitrary large values of $t$ with  $$\Bigl | \frac{1}{\psi_\alpha(t)} - \frac{1}{\psi_\beta(t)} \Bigl | \geqslant \sqrt5\left(1-\sqrt{\frac{\sqrt5-1}{2}}\right)t.$$
The constant on the right-hand side is optimal.

\end{abstract}

\section{Introduction}

For irrational number $\xi\in\mathbb{R}$ we consider irrationality measure function
$$ \psi_\xi(t) = \min_{1\le q\le t,\, q\in\mathbb{Z}} \| q\xi \| ,$$
where $|| . || $ denotes distance to the nearest integer.\\
Let
$$ Q_0\le Q_1<Q_2<\ldots<Q_n<Q_{n+1}<\ldots $$
be the sequence of denominators of convergents to $\xi$. It it a well-known fact (see \cite{perron} or \cite{schmidt}) that

\begin{equation}\label{minimum}
 \psi_\xi(t) = \| Q_n\xi \| \text{  for  } Q_n\le t < Q_{n+1}. 
 \end{equation}
It is clear for any $\xi\in\mathbb{R}$ and for every $t\ge1$ we have

\begin{equation}\label{minkow}
\psi_\xi(t) \le \frac{1}{t}. 
 \end{equation}
 
In \cite{kan}, Kan and Moshchevitin proved that for any two irrational numbers $\alpha, \beta$ satisfying $\alpha\pm\beta\notin\mathbb{Z}$ the difference 
$$ \psi_\alpha(t) - \phi_\beta(t)  $$
changes its sign infinitely many times as $t\to\infty$.

Obviously, it follows that the difference of their reciprocals 
\begin{equation}\label{d(t)}
 d(t)=d_{\alpha,\beta}(t) = \frac{1}{\psi_\beta(t)}-\frac{1}{\psi_\alpha(t)}
\end{equation}
also changes sign infinitely many times as $t\to\infty$.

In 2017 Dubickas \cite{dubic} proved the following result.\\

\begin{theorem}\label{TheoremA}	For any irrational numbers $\alpha, \beta$ satisfying $\alpha\pm\beta\notin\mathbb{Z}$ the sequence $d(n),n\in\mathbb{Z_+}$ is unbounded. 
 \end{theorem}

To formulate further results, we need constants
 \begin{equation}\label{constTAU}
 \tau = \frac{\sqrt5 +1}{2},
  \end{equation}
  
  \begin{equation}\label{constPHI}
 \phi = \frac{\sqrt5 -1}{2},
  \end{equation} 
 
  \begin{equation}\label{constK}
 K=\sqrt\tau-1=0.2720^+,
  \end{equation}

  \begin{equation}\label{constC}
 C = K(\sqrt{\tau}+\tau^{-3/2})=\sqrt5(1-\sqrt{\phi})=0.47818^+.
 \end{equation}


In 2019 Moshchevitin \cite{mosh} proved the following

\begin{theorem}\label{TheoremB}		Let $\alpha$ and $\beta$ be the two irrational numbers. If $\alpha\pm\beta\notin\mathbb{Z}$, then for every $T\ge 1$ there exists $t\ge T$ such that

\begin{equation}\label{mosh1}
 | \psi_\alpha(t) - \psi_\beta(t) | \geqslant K\cdot \min( \psi_\alpha(t), \psi_\beta(t) ). 
\end{equation}
\end{theorem}
It was also proved that constant $K$ in \eqref{mosh1} is optimal.\\
Using Theorem \ref{TheoremB} and \eqref{minkow} Moshchevitin deduced stronger version of Theorem \ref{TheoremA}.\\

\begin{coroll}			Let $\alpha$ and $\beta$ be the two irrational numbers. If $\alpha\pm\beta\notin\mathbb{Z}$, then for every $T\ge 1$ there exists $t\ge T$ such
\begin{equation}\label{mosh2}
 \Bigl | \frac{1}{\psi_\alpha(t)} - \frac{1}{\psi_\beta(t)} \Bigl | \geqslant Kt, \text{\,\,\, with $K$ from \eqref{constK}} . 
\end{equation}
\end{coroll}
In this paper we give an improvement of this result with the best possible constant.

\begin{theor}
			1) Let $\alpha$ and $\beta$ be the two irrational numbers. If $\alpha\pm\beta\notin\mathbb{Z}$, then for every $T\ge 1$ there exists $t\ge T$ such
\begin{equation}\label{shul}
 \Bigl | \frac{1}{\psi_\alpha(t)} - \frac{1}{\psi_\beta(t)} \Bigl | \geqslant Ct, \text{\,\,\, where $C$ is defined in \eqref{constC}} . 
\end{equation}
2) Constant $C$ in \eqref{shul} is optimal, that is for any $\varepsilon>0$, there exist $\alpha$ and $\beta$ with $\alpha\pm\beta\notin\mathbb{Z}$ such that 
$$
\Bigl | \frac{1}{\psi_\alpha(t)} - \frac{1}{\psi_\beta(t)} \Bigl | \leqslant (C+\varepsilon) t
$$
for all $t$ large enough.
\end{theor}

\section{Auxiliary results}
Throughout, we consider $\alpha\notin\mathbb{Q}$ and its continued fraction expansion

\begin{equation}
\alpha =  a_0 + \cfrac{1}{a_1+\cfrac{1}{a_2+\cdots}}=[a_0;a_1,a_2,\ldots], \,\,\, a_j \in \mathbb{Z}_+
\end{equation}
By $\alpha_r$ we denote a tail of continued fraction
\begin{equation}
\alpha_r=  [a_r;a_{r+1},a_{r+2},\ldots]
\end{equation}
Similarly, let $\beta = [b_0; b_1,b_2,\ldots]$ and $\beta_m = [b_m; b_{m+1}, b_{m+2}, \ldots]$.\\
For irrational numbers $\alpha$ and $\beta$ we denote the denominators of their convergents by $q_n$ and $t_m$ respectively.\\
Define 
$$\xi_n = \psi_\alpha( q_n ) \text{\,\,\,\,\, and \,\,\,\,\,\,} \eta_s = \psi_\beta ( t_s ).$$
It is well known that 
\begin{equation}
q_n\alpha - \frac{(-1)^n}{q_n\alpha_{n+1}+q_{n-1}}\in\mathbb{Z}.
\end{equation}
Using this fact as well as \eqref{minimum} we see that

\begin{equation}\label{defal}
\frac{1}{\psi_\alpha(n)} = q_r\alpha_{r+1}+q_{r-1}=q_{r+1}+\frac{q_r}{\alpha_{r+2}}
\end{equation}
where  $r\geqslant0$ is the largest integer satisfying $q_r\leqslant n$.

Similarly,
\begin{equation}\label{defbe}
\frac{1}{\psi_\beta(n)} = t_l\beta_{l+1}+t_{l-1}=t_{l+1}+\frac{t_l}{\beta_{l+2}}
\end{equation}
where  $l\geqslant0$ is the largest integer satisfying $t_l\leqslant n$.

From now on we consider only irrational numbers $\alpha$ and $\beta$ with $\alpha\pm\beta\notin\mathbb{Z}$.\\
The following two simple lemmas in fact were used in \cite{mosh}. Here we formulate them without a proof.
\begin{lem}\label{conseq}
There exist only finitely many pairs $(q_n,q_{n+1})=(t_m,t_{m+1})$.
\end{lem}

\begin{lem}\label{conseq=1}
There exist only finitely many pairs $(q_n,q_{n+2})=(t_{m+1},t_{m+2})$ with $a_{n+2}=1$.
\end{lem}
Our next auxiliary result is the following
\begin{lem}\label{xd}
\,
\\
If $\eta_s \in (\xi_n, \xi_{n-1}) $, then either
$$ \frac{1}{\eta_s}  - \frac{1}{\xi_{n-1}} \ge  t_s \left( \beta_{s+1} + \frac{t_{s-1}}{t_s} \right) \left( 1- \frac{1}{\sqrt{\alpha_{n+1}}} \right)$$
or
$$ \frac{1}{\xi_n}  - \frac{1}{\eta_s} \ge  q_n \left( \alpha_{n+1} + \frac{q_{n-1}}{q_n} \right) \left( 1- \frac{1}{\sqrt{\alpha_{n+1}}} \right).$$

\end{lem}

\begin{proof}
It is a well-known fact, that for any $n\in\mathbb{N}$ one has
$$ \frac{ \xi_{n-1}}{\xi_{n}} = \alpha_{n+1} .$$
As $$  \frac{ \xi_{n-1}}{\eta_s}  \cdot  \frac{ \eta_s}{\xi_{n}} = \frac{ \xi_{n-1}}{\xi_{n}} = \alpha_{n+1}, $$
we see that either $ \frac{ \xi_{n-1}}{\eta_s} \ge \sqrt{ \alpha_{n+1} }$ or $ \frac{ \eta_s}{\xi_{n}}  \ge \sqrt{ \alpha_{n+1} }$.

In the first case, one has 
$$ \frac{1}{\eta_s}  - \frac{1}{\xi_{n-1}} \ge  \frac{1}{\eta_s}  -  \frac{1}{\sqrt{\alpha_{n+1}} \eta_s} = \frac{1}{\eta_s} \left( 1- \frac{1}{\sqrt{\alpha_{n+1}}} \right) =  t_s \left( \beta_{s+1} + \frac{t_{s-1}}{t_s} \right) \left( 1- \frac{1}{\sqrt{\alpha_{n+1}}} \right). $$
Last equality is due to \eqref{defbe}.\\
Similarly, in the case $ \frac{ \eta_s}{\xi_{n}}  \ge \sqrt{ \alpha_{n+1} }$, using \eqref{defal} we get
$$ \frac{1}{\xi_n}  - \frac{1}{\eta_s} \ge  \frac{1}{\xi_n}   -  \frac{1}{\sqrt{\alpha_{n+1}} \xi_n} = \frac{1}{\xi_n} \left( 1- \frac{1}{\sqrt{\alpha_{n+1}}} \right) =  q_n \left( \alpha_{n+1} + \frac{q_{n-1}}{q_n} \right) \left( 1- \frac{1}{\sqrt{\alpha_{n+1}}} \right).$$
\end{proof}

\,
By $F_n$ we denote  $n$th Fibonacci number, that is 
$$F_{1} = F_2=1, F_{n+1}=F_n+F_{n-1}.$$
By $\langle a_1, \ldots, a_n \rangle $ we denote a denominator of the continued fraction $[0;a_1,\ldots, a_n]$.
Using this notation, we recall that
\begin{equation}\label{classic} \langle a_1,\ldots, a_n, b_1,\ldots, b_k \rangle = \langle a_1,\ldots a_n \rangle \langle b_1\ldots b_k \rangle + \langle a_1,\ldots a_{n-1} \rangle \langle b_2\ldots b_k \rangle.
\end{equation}

\section{Proof of the main inequality}
In this section we prove statement 1) of Theorem 1.

For a contradiction, assume that for some $\alpha$ and $\beta$ with $\alpha\pm\beta\notin\mathbb{Z}$ there exists $T\in\mathbb{Z}_+$ such that for $d(t)$ defined in \eqref{d(t)} we have

\begin{equation}\label{contr}
|d(t)| < Ct 
\end{equation}
for each $t>T $.

\begin{lem}\label{lemma=1}
Suppose that $$q_{n-1}\leqslant t_{m-1} < q_n < t_{m}$$ 
and \eqref{contr} holds for all $t\ge t_{m-1} $, then
\begin{equation}\label{first}
a_{n+1}=1.
\end{equation}
Similarly, if $$t_{m-1}\leqslant q_{n-1} < t_m  < q_n$$
and \eqref{contr} holds for all $t\ge q_{n-1} $, then
\begin{equation}\label{second}
b_{m+1}=1.
\end{equation}
\end{lem}
\begin{proof}
Let us prove \eqref{first}. Using \eqref{d(t)}, \eqref{defal} and \eqref{defbe} we obtain
$$d(t_{m-1}) = t_m + \frac{t_{m-1}}{\beta_{m+1}} - q_n - \frac{q_{n-1}}{\alpha_{n+1}}$$
and
$$d(q_n) = t_m + \frac{t_{m-1}}{\beta_{m+1}} - q_{n+1} - \frac{q_{n}}{\alpha_{n+2}}.$$
Subtracting the second formula from the first one we get
$$ 2C q_n > d(t_{m-1}) - d(q_n) = q_{n+1} + \frac{q_{n}}{\alpha_{n+2}} - q_n - \frac{q_{n-1}}{\alpha_{n+1}} $$
$$ > q_n \Bigl( a_{n+1} -1 + \frac{1}{\alpha_{n+2}} \Bigl) + q_{n-1} \Bigl ( 1- \frac{1}{\alpha_{n+1}} \Bigl) >q_n ( a_{n+1} -1 ).$$
Hence, $a_{n+1} < 2C + 1 = 1.95636^+$. As $a_{n+1} \in\mathbb{N}$, we see that $a_{n+1}=1$. The proof of \eqref{second} is the same.
\end{proof}
To prove statement 1) of Theorem 1 we consider two cases.\\
Case 1. Either there exist infinitely many $n$, such that $a_{n+1}\ge 2$ or infinitely many $m$, such that $b_{m+1} \ge 2$. \\
Case 2. Both $\alpha$ and $\beta$ are equivalent to $\tau$.\\

In Case 1 we consider several subcases.\\
1.1) There exist infinitely many $a_{n+1}\ge 2$, such that $q_{n}$ is not a denominator of convergent to $\beta$ or there exist infinitely many $b_{m+1} \ge 2$, such that $t_{m}$ is not a denominator of convergent to $\alpha$.\\
Without loss of generality suppose that this condition holds for partial quotients and denominators of $\alpha$.\\
Consider $a_{n+1}\ge2$, such that $q_n$ is not a denominator of convergent to $\beta$ and \eqref{contr} being true for $t\ge q_{n-1}$.\\
 Consider indices $n$ and $m$ such that $t_{m-1}<q_n<t_m$. Suppose that one of the 
 $$
 q_{n-1}=t_{m-1} < q_n < t_{m}
\text{\,\,\,\,or\,\,\,\,\,\,\,} q_{n-1}< t_{m-1}  < q_n < t_{m} $$ 
  is true. Then by Lemma \ref{lemma=1} $a_{n+1}$ has to be equal to 1. There is a contradiction with $a_{n+1} \ge 2$.  Hence the only possible case is 
 \begin{equation}\label{1case}
 t_{m-1}  < q_{n-1} < q_n < t_{m}. 
 \end{equation}

Now using \eqref{d(t)} and \eqref{1case} we obtain
$$
d(q_{n-1}) = t_{s+1} + \frac{t_s}{\beta_{s+2}} - q_n -\frac{q_{n-1}}{\alpha_{n+1}}
$$
and
$$
d(q_{n}) = t_{s+1} + \frac{t_s}{\beta_{s+2}} - q_{n+1} -\frac{q_{n}}{\alpha_{n+2}}.
$$
By subtracting the second formula from the first we get
$$
d(q_{n-1}) - d(q_{n})  =  q_{n+1} + \frac{q_{n}}{\alpha_{n+2}} - q_n -\frac{q_{n-1}}{\alpha_{n+1}}
$$
$$
= q_n \Bigl( a_{n+1} - 1 + \frac{1}{\alpha_{n+2}} \Bigl) + q_{n-1} \Bigl( 1-\frac{1}{\alpha_{n+1}}  \Bigl) > q_n.
$$
The last inequality is due to $a_{n+1}\ge 2$. As \eqref{contr} is true for $t\ge q_{n-1}$, we can estimate
$$
d(q_{n-1}) - d(q_{n})  < C q_{n-1} + C q_n < 2C q_n.
$$
We come to a contradiction $ 2C q_n > q_n$, as $2C = 0.95636^+$.\\
1.2) There exists $N$, such that for every $a_{n+1} \ge 2$ with $n>N$, $q_n$ is a denominator of convergent to $\beta$, for every $b_{n+1}\ge2$ with $n>N$, $t_n$ is a denominator of convergent to $\alpha$ and \eqref{contr} is true for $t \ge \min(q_N,t_N)$.\\
Without loss of generality suppose that there exist infinitely many $a_{n+1}\ge2$ for $\alpha$.\\
Consider $a_{n+1}\ge2$ with $n>N$. Then $q_n = t_s$ for some index $s$.

Next, for the local convenience we introduce the union $U=D_{\alpha} \cup D_{\beta}$ of two sequences $D_\alpha = \{ q_0 = 1 \le q_1 < q_2 < \ldots \}$ and $D_\beta = \{ t_0 = 1 \le t_1 < t_2 < \ldots \} $. We construct an infinite word $W$ on an alphabet $\{ B^*_*, Q^*, T_* \}$, where $*$ will match some indices from $D_\alpha$ and $D_\beta$, by following procedure. The first element of $U$ is 1. It belongs to both sequences, so we start our infinite word with $B^0_0$. Then, we write the letter $Q^j$ if the next element of $U$ is $q_j$ from the set $D_\alpha$, but not from the set $D_\beta$. We write letter $T_i$ if it is $t_i$ from the set $D_\beta$, but not from the set $D_\alpha$. Also, we write $B^j_i$ if the next element of $U$ is present in both $D_\alpha$ and $D_\beta$ as $q_j$ and $t_i$ respectively.

As $q_n=t_s$, the word $W$ contains letter $B^n_s$. We consider possible cases of the location of the next letter $B^*_*$.
\\
1.2.1) The word $W$ can not contain $B^n_s B^{n+1}_{s+1}$ as a subword infinitely many times, since it is forbidden by Lemma \ref{conseq}.\\
1.2.2) The word $W$ can not contain subwords $B^n_s Q^{n+1} B^{n+2}_{s+1}$ or $B^n_s T_{s+1} B^{n+1}_{s+2}$ infinitely many times as it is forbidden by Lemma \ref{lemma=1} and Lemma \ref{conseq=1}. Indeed, by Lemma \ref{lemma=1}, if $W$ has a subword $B^n_s Q^{n+1}$ or $B^n_s T_{s+1} $, then $a_{n+2}=1$ or $b_{s+2}=1$ respectively. Both cases are covered by Lemma \ref{conseq=1}, which states that there exist only finitely many subwords $B^n_s Q^{n+1} B^{n+2}_{s+1}$ or $B^n_s T_{s+1} B^{n+1}_{s+2}$ with $a_{n+2}=1$ or $b_{s+2}=1$ respectively. So there are not such subwords for $n,s$ large enough.\\
1.2.3a) Let the word $W$ contain a subword $B^n_s Q^{n+1} \ldots Q^{n+k-1} B^{n+k}_{s+1}$ with $k\ge3$. Then 
$$t_{s+1}\ge q_{n+3} = q_{n+2} + q_{n+1}=2q_{n+1}+q_n$$ and by \eqref{d(t)} we have 
$$
d(q_n)=t_{s+1} - q_{n+1} + \frac{t_s}{\beta_{s+2}}-\frac{q_n}{\alpha_{n+2}}
$$

$$
\ge q_{n+3} - q_{n+1} + q_n \Bigl( \frac{1}{\beta_{s+2}}-\frac{1}{\alpha_{n+2}} \Bigl)
$$

$$
=q_{n+1} + q_n \Bigl( 1+ \frac{1}{\beta_{s+2}}-\frac{1}{\alpha_{n+2}} \Bigl) > q_{n+1} > a_{n+1} q_n \ge 2 q_n.
$$
As $|d(q_n)|< C q_n$, where  $C = \sqrt5(1-\sqrt{\phi})=0.47818^+$, we come to a contradiction with $d(q_n)>2q_n$.\\
1.2.3b) Let the word $W$ contain a subword $B^n_s T_{s+1} \ldots T_{s+m-1} B^{n+1}_{s+m}$ with $m\ge3$. Then 
$$q_{n+1}\ge t_{s+3} = t_{s+2} + t_{s+1}=2t_{s+1}+t_s$$
by \eqref{d(t)} we have 
$$
-d(t_s)=q_{n+1} -t_{s+1}  + \frac{q_n}{\alpha_{n+2}} - \frac{t_s}{\beta_{s+2}}
$$

$$
\ge t_{s+3} - t_{s+1} + t_s \Bigl( \frac{1}{\beta_{s+2}}-\frac{1}{\alpha_{n+2}} \Bigl)
$$

$$
=t_{s+1} + t_s \Bigl( 1+ \frac{1}{\beta_{s+2}}-\frac{1}{\alpha_{n+2}} \Bigl) > t_{s+1} > t_s.
$$
As $|d(t_s)|< C t_s$, where  $C = \sqrt5(1-\sqrt{\phi})=0.47818^+$, we come to a contradiction with $-d(t_s)>t_s$.\\
1.2.4) The word $W$ contains subword $B^n_s  \ldots B^{n+k+1}_{s+m+1}$, where between $B^n_s$ and  $B^{n+k+1}_{s+m+1}$ there are $k\ge1$ letters $Q^*$ and $m\ge1$ letters $T_*$ in any order.

We know that 
$$
 \langle a_1, \ldots , a_n \rangle =  q_n = t_s = \langle b_1, \ldots , b_s \rangle 
 $$
$$\text{and}$$
$$
 \langle a_1, \ldots, a_{n+k+1} \rangle = q_{n+k+1} = t_{s+m+1} =  \langle b_1, \ldots, b_{s+m+1} \rangle .
$$
By the conditions of the case 1.2), if there is a partial quotient $a_{n+1}\ge2$ or $b_{n+1}\ge2$ for $n>N$, then the word $W$ contains a letter $B^{n}_*$ or $B^*_n$ respectively. As we consider letter $B^n_s$ along with the closest to it letter $B^*_*$, we deduce that there are only partial quotients equal to $1$ in sequences $(a_{n+2},\ldots,a_{n+k+1})$ and $(b_{s+2},\ldots,b_{s+m+1})$. So we have
$$
q_{n+k+1} = \langle a_1, \ldots, a_n, a_{n+1}, \underbrace{1, \ldots, 1}_{k\ge1 \text{ times}} \rangle \text{, \,\,\, where \,\,\,}  a_{n+1}\ge 2
$$
and
$$
t_{s+m+1} = \langle b_1, \ldots, b_s, b_{s+1}, \underbrace{1, \ldots, 1}_{m\ge1 \text{ times}}  \rangle .
$$
Using formula \eqref{classic}, we have
$$
\langle a_1, \ldots a_{n+1} \rangle \langle \underbrace{1, \ldots, 1}_{k \text{ times}} \rangle + 
\langle a_1, \ldots a_{n} \rangle \langle \underbrace{1, \ldots, 1}_{k-1 \text{ times}} \rangle 
$$
$$
=\langle b_1, \ldots b_{s+1} \rangle \langle \underbrace{1, \ldots, 1}_{m \text{ times}} \rangle + 
\langle b_1, \ldots b_{s} \rangle \langle \underbrace{1, \ldots, 1}_{m-1 \text{ times}} \rangle .
$$
Or, for short,

$$
q_{n+1}F_{k+1} + q_n F_{k} = t_{s+1} F_{m+1} + t_s F_m .
$$
From this equality we can express $t_{s+1}$ as 
\begin{equation}\label{t_{s+1}}
t_{s+1} = q_{n+1} \frac{F_{k+1}}{F_{m+1}} + q_n \frac{F_k-F_m}{F_{m+1}}.
\end{equation}
Now, if $t_{s+1} > q_{n+1}$, then $m\le k-1$. Indeed, assume that $m \ge k$. Then from \eqref{t_{s+1}} we have
$$
 t_{s+1} = q_{n+1} \frac{F_{k+1}}{F_{m+1}} + q_n \frac{F_k-F_m}{F_{m+1}} \le 
q_{n+1} \frac{F_{m+1}}{F_{m+1}} + q_n \frac{F_m-F_m}{F_{m+1}} =  q_{n+1} .
$$
Next we estimate tails of continued fractions of $\alpha$ and $\beta$ as
\begin{equation}\label{tailA}
\frac{1}{\beta_{s+2} }= [b_{s+2}, b_{s+3},\ldots] > \frac{1}{b_{s+2}+1}= \frac{1}{2},
\end{equation}
\begin{equation}\label{tailB}
\frac{1}{\alpha_{n+2} }= [a_{n+2}, a_{n+3},\ldots] < \frac{1}{a_{n+2}+\frac{1}{a_{n+3}+1
}} = \frac{1}{1+\frac{1}{1+1}} =\frac{2}{3}.
\end{equation}
We apply \eqref{t_{s+1}},\eqref{tailA}, \eqref{tailB}, the fact that $m\le k-1$ and $q_{n+1}>a_{n+1}q_n$ to get
$$
d(q_n)=t_{s+1} - q_{n+1} + \frac{t_s}{\beta_{s+2}}-\frac{q_n}{\alpha_{n+2}}
$$
$$
=q_{n+1} \Bigl( \frac{F_{k+1} - F_{m+1}}{F_{m+1}} \Bigl) + q_n \Bigl( \frac{F_k-F_m}{F_{m+1}} + \frac{1}{\beta_{s+2}}-\frac{1}{\alpha_{n+2}}  \Bigl)
$$
$$
\ge (a_{n+1}-1)q_n\frac{F_m}{F_{m+1}}+ q_n \Bigl( 1 + \frac{1}{\beta_{s+2}}-\frac{1}{\alpha_{n+2}} \Bigl) 
>  \frac{3a_{n+1}+2}{6} q_n \ge  \frac{4}{3} q_n .
$$
As before, we come to a contradiction with $d(q_n)< C q_n = 0.47818^+ q_n $.\\
If $t_{s+1} < q_{n+1}$, then we use 
$$
t_{s+1} = q_{n+1} \frac{F_{k+1}}{F_{m+1}} + q_n \frac{F_k-F_m}{F_{m+1}},
$$
$$
m \ge k+1,
$$
$$
\frac{1}{\alpha_{n+2} }= [a_{n+2}, a_{n+3},\ldots] > \frac{1}{a_{n+2}+1}= \frac{1}{2},
$$
$$
\frac{1}{\beta_{s+2} }= [b_{s+2}, b_{s+3},\ldots] < \frac{1}{b_{s+2}+\frac{1}{b_{s+3}+1
}} = \frac{1}{1+\frac{1}{1+1}} =\frac{2}{3}
$$
to get
$$
C q_n > - d(q_n)=  q_{n+1} - t_{s+1} +\frac{q_n}{\alpha_{n+2}}  - \frac{t_s}{\beta_{s+2}}$$
$$ = q_{n+1} \frac{F_{m+1}-F_{k+1}}{F_{m+1}} + q_n \left( \frac{F_m -F_k}{F_{m+1}} + \frac{1}{\alpha_{n+2}}-\frac{1}{\beta_{s+2}} \right)$$
$$ >a_{n+1} q_n \frac{ F_{m-1}}{F_{m+1}} +q_n \left( \frac{F_{m-2}}{F_{m+1}}-\frac{1}{6} \right) 
 >q_n \left ( \frac{ 2F_{m-1}+F_{m-2}}{F_{m+1}}-\frac{1}{6} \right)=  \frac{5}{6} q_n .
$$
2) There exists $N_0$, such that for every $n,m > N_0$ one has $a_{n+1} = 1$ and $b_{m+1}=1$. 
This means that $\beta_{n+2} = \alpha_{n+2} = \tau $ for any $n > N_0$. Let us construct a word $W$ as in the case 1.2).

\begin{lem}\label{qqtt}There are only finitely many subwords $Q^n Q^{n+1}$ or $T_s T_{s+1}$ in word $W$ under the conditions of the case 2).
\end{lem}
\begin{proof}
Assume that $W$ has infinitely many subwords $Q^n Q^{n+1}$. Then, it has infinitely many subwords $T_s Q^n Q^{n+1}$ or $B_s^{n-1} Q^n Q^{n+1}$. In both cases we have
$$ q_{n-1} \le t_s < q_n < q_{n+1} < t_{s+1}, $$
so we can deduce
$$ - d(q_{n+1}) = q_{n+1} \alpha_{n+2} + q_n - t_s\beta_{s+1} - t_{s-1}$$
and
$$ - d(t_s) = q_{n} +\frac{q_{n-1}}{\alpha_{n+1}}- t_s\beta_{s+1} - t_{s-1}.$$
By subtracting the second equality from the first we obtain
$$ C(q_{n+1}+ q_n) >C(q_{n+1}+ t_s) > d(t_s) - d(q_{n-1}) = q_{n+1} \alpha_{n+2}  - \frac{q_{n-1}}{\alpha_{n+1}} = \tau q_{n+1} - \phi q_{n-1},$$
or
$$ q_{n+1} <q_{n+1} ( \tau - C) < C q_n + \phi q_{n-1} < \phi ( q_n + q_{n-1} ) =\phi q_{n+1},$$
which is a contradiction. 
In case $W$ has infinitely many subwords $T_s T_{s+1}$ the proof is the same.
\end{proof}
Now let us show that there are only finitely many letters $B^*_*$. Assume that the word $W$ contains letter $B^n_s$ for $n,s$ large enough. Then by Lemma \ref{conseq} the next letter is either $Q^{n+1}$ or $T_{s+1}$. Without loss of generality suppose that it is $Q^{n+1}$. Then by Lemma \ref{conseq=1} and Lemma \ref{qqtt} the next letter is $T_{s+1}$. If the next letter is $B_{s+2}^{n+2}$, then, as $b_{s+2}=a_{n+2}=1$, we have
$$ q_{n+1} + q_n = q_{n+2}= t_{s+2} = t_{s+1}+ t_s $$
and as $q_n = t_s$, we get $q_{n+1} = t_{s+1}$, which is a contradiction.
$B^n_s  Q^{n+1} T_{s+1} T_{s+2}$ is forbidden by Lemma \ref{qqtt}, hence we always have a subword $B^n_s  Q^{n+1} T_{s+1} Q^{n+2}$. Now we can show that a subword $Q^{i}T_{j} Q^{i+1} B_{j+1}^{i+2}$ 
is also forbidden for any $i>n,j>s$, since in the equality
$$t_j+t_{j-1}=t_{j+1}=q_{i+2} = q_{i+1} + q_i $$
the terms on the right side are greater than the corresponding terms on the left. The same argument works with a subword $T_{j-1} Q^{i}T_{j}  B_{j+1}^{i+1}$.

We have shown that for $n$ large enough, $W$ has an infinite subword $Q^{n-1} T_s Q^{n} T_{s+1} \ldots $.\\
In our case this also implies that 
$$
\xi_{n-1} > \eta_s > \xi_{n} > \eta_{s+1} > \xi_{n+1}>\ldots.
$$

At this point we use Lemma \ref{xd} twice: for $\eta_s \in (\xi_n, \xi_{n-1}) $ and for $\xi_n \in (\eta_{s+1}, \eta_{s}) $. Then, if $ \frac{ \eta_s}{\xi_{n}}  \le \sqrt{ \tau}$, we have 
$$
 \frac{ \xi_{n-1}}{\eta_s} \ge \sqrt{\tau } \text{ \,\,\,\,\, and  \,\,\,\,\,  }
  \frac{ \xi_{n}}{\eta_{s+1}} \ge \sqrt{ \tau }.
  $$
  Hence both 
  \begin{equation}\label{option1} d(t_s)=\frac{1}{\eta_s}  - \frac{1}{\xi_{n-1}} \ge    t_s \left( \tau + \frac{t_{s-1}}{t_s} \right) \left( 1- \sqrt{\phi} \right) 
  \end{equation}
  and
    \begin{equation}\label{option2} d(t_{s+1})=\frac{1}{\eta_{s+1}}  - \frac{1}{\xi_{n}} \ge   t_{s+1} \left( \tau + \frac{t_{s}}{t_{s+1}} \right) \left( 1- \sqrt{\phi} \right)
    \end{equation}
  are true.
  As $\beta \sim \tau$, we know that $\frac{t_s}{t_{s+1}}\to\phi$ as $s\to\infty$. Moreover, one of $\frac{t_s}{t_{s+1}}$ and $\frac{t_{s-1}}{t_{s}}$ is less than $\phi$, and another one is greater than $\phi$ for every $s$ big enough. Now from \eqref{option1} and \eqref{option2} we have that
  $$
  d(t_s)>t_s  \sqrt5(1-\sqrt{\phi})= t_s C \text{\,\,\,\,\,\,or\,\,\,\,\,\,}   d(t_{s+1})>t_{s+1}  \sqrt5(1-\sqrt{\phi})= t_{s+1} C.
  $$
   In case $ \frac{ \eta_s}{\xi_{n}}  \ge \sqrt{ \tau}$ we get that $ \frac{ \eta_{s+1}}{\xi_{n+1}}  \ge \sqrt{ \tau}$ and similarly we will come to one of
   $$
  -d(q_n)>q_n  \sqrt5(1-\sqrt{\phi})= q_n C \text{\,\,\,\,\,\,or\,\,\,\,\,\,}   -d(q_{n+1})>q_{n+1}  \sqrt5(1-\sqrt{\phi})= q_{n+1} C
  $$
  being true by the same argument.

In every case we came to a contradiction, hence the assumption \eqref{contr} was wrong. \\

\section{Optimality of Theorem 1}
Now let us prove the second statement of Theorem 1, that is the constant $C$ is optimal. We consider two irrational numbers, $\tau$ from \eqref{constTAU} and
$\theta$, which we will define below.\\
First, for every $\varepsilon > 0$ there exist $U,V \in\mathbb{Z}$, such that
$$ \left | V + \frac{U}{\tau} - \sqrt{\tau} \right | < \varepsilon .$$
Define the sequence $X_n$ as
$$ X_0 = U, \,\,\, X_1 = V, \,\,\, X_{n+1} = X_n + X_{n-1} \text{\, for every\,\,\,} n\ge 1. $$
Then we have
$$ X_n = A \tau^n + B (-\tau)^{-n} \text{,\,\,\,\,\,where\,\,\,\,\,} A=\frac{\tau V + U}{\tau+2}.$$
As $A>0$, there exists $k\in\mathbb{Z_+}$, such that $X_{k-1}, X_k \geqslant 1$. Let
$$\frac{X_{k-1}}{X_k} = [0; b_w,\ldots, b_1], \,\,\,\,\,\,\,\, b_j \in\mathbb{Z_+}$$
and define $\theta$ as 
$$\theta = [0; b_1,\ldots, b_w, \overline{1} ].$$
We can show that 
\begin{equation}\label{keps}
\Bigl | \frac{1}{\psi_\tau(t)} - \frac{1}{\psi_\theta(t)} \Bigl | \leqslant (C+\varepsilon) t
\end{equation}
for all $t$ large enough.\\
Denominator $q_n$ of convergent $p_n/q_n$ to $\tau$ is equal to
$$ q_n = F_n = \frac{1}{\sqrt5} ( \tau^n - (-\tau)^{-n} ).$$
Denominator $s_n$ of convergent $r_n/s_n$ to $\theta$ is equal to
$$ s_n = X_{n-n_0}, \,\,\,\,\, n_0 = k-w $$
for $n$ large enough. Note, that 
$$ \frac{X_n}{F_n} \sim A\sqrt5 = V+ U \tau^{-1}, \,\,\,\,\, | A\sqrt5 - \sqrt\tau | < \varepsilon. $$
We use
$$ q_n = F_n < X_n = s_{n+n_0} < F_{n+1} = q_{n+1} $$
and 
$$ \frac{1}{\psi_\tau(t)} = \frac{1}{\xi_n} = \frac{1}{|| q_n \tau || }\sim \sqrt5 F_n, \,\,\,\,F_n \leqslant t < F_{n+1}, $$
$$ \frac{1}{\psi_\theta(t)} = \frac{1}{\eta_n} = \frac{1}{|| s_{n+n_0} \eta ||}
 \sim \sqrt5 X_n, \,\,\,\,X_n \leqslant t < X_{n+1} $$
 so that
 $$ \frac{\xi_{n-1}}{\eta_n} \sim \frac{F_n}{X_{n-1}} = \sqrt\tau + O(\varepsilon), \,\,\,\,\, \frac{\eta_{n}}{\xi_n} \sim \frac{X_n}{F_{n}} = \sqrt\tau + O(\varepsilon),$$
 and we have
$$ \Bigl | \frac{1}{\psi_\tau(t)} - \frac{1}{\psi_\theta(t)} \Bigl | \sim X_n \sqrt5 \left | 1- \frac{1}{\sqrt\tau} \right |, \,\,\,\,\,\,\,\, X_n \leqslant t<F_{n+1},$$
$$ \Bigl | \frac{1}{\psi_\tau(t)} - \frac{1}{\psi_\theta(t)} \Bigl | \sim F_{n+1} \sqrt5 \left | 1- \frac{1}{\sqrt\tau} \right |, \,\,\,\,\,\,\,\, F_{n+1} \leqslant t<X_{n+1}. \,\,\,\,\,\,\,\square $$

\section*{Acknowledgements}

The author thanks Nikolay Moshchevitin for careful reading and helpful comments.

N. Shulga is a scholarship holder of "BASIS" Foundation for Development of Theoretical Physics and Mathematics and is supported in part by the Moebius Contest Foundation for Young Scientists.

\end{document}